\documentclass[12pt,reqno]{amsart}
\usepackage{amssymb}

\newcommand{\la}{\lambda}
\newcommand{\al}{\alpha}

\newcommand{\fy}{\varphi}

\newcommand{\p}{\partial}
\newcommand{\I}{\infty}

\newcommand{\ti}{\tilde}
\newcommand{\R}{\mathbb{R}}

\newcommand{\Z}{\mathbb{Z}}
\newcommand{\F}{\mathcal{F}}
\renewcommand{\S}{\mathcal{S}}

\renewcommand{\hat}{\widehat}

\numberwithin{equation}{section}

\newtheorem{thm}{Theorem}%[section]

\newtheorem{lem}[thm]{Lemma}

\theoremstyle{remark}
\newtheorem{rem}{Remark}
\newtheorem{defn}{Note}

\newcommand{\ran}{\rangle}
\newcommand{\lan}{\langle}

\newcommand{\lec}{\lesssim}

\newcommand{\EQ}[1]{\begin{equation} \begin{split} #1 \end{split} \end{equation}}
\newcommand{\Del}[1]{}
\newcommand{\CAS}[1]{\begin{cases} #1 \end{cases}}
\newcommand{\pt}{&}
\newcommand{\pr}{\\ &}
\newcommand{\pq}{\quad}
\newcommand{\pn}{}

\newcommand{\LR}[1]{{\lan #1 \ran}}
\newcommand{\de}{\delta}

\newcommand{\ta}{\tau}
\renewcommand{\t}{\tau}

\newcommand{\x}{\xi}

\newcommand{\na}{\nabla}

\newcommand{\supp}{\operatorname{supp}}

\newcommand{\De}{\Delta}

\begin{document}
\title{Global well-posedness and scattering for Skyrme wave maps}

\author{Dan-Andrei Geba, Kenji Nakanishi, and Sarada G. Rajeev}

\address{Department of Mathematics, University of Rochester, Rochester, NY 14627, U.S.A.}
\email{dangeba@math.rochester.edu}
\address{Department of Mathematics, Kyoto University, Kyoto 606-8502, Japan}
\email{n-kenji@math.kyoto-u.ac.jp}
\address{Department of Physics and Astronomy, Department of Mathematics, University of Rochester, Rochester, NY 14627, U.S.A.}
\email{rajeev@pas.rochester.edu}
\date{}

\begin{abstract}
We study equivariant  solutions for two models (\cite{S1}-\cite{S3}, \cite{AN}) arising in high energy physics, which are generalizations of the wave maps theory (i.e., the classical nonlinear $\sigma$ model) in $3+1$ dimensions. We prove global existence and scattering for small initial data in critical Sobolev-Besov spaces. 
\end{abstract}

\subjclass[2000]{35L70, 81T13}
\keywords{Wave maps, Skyrme model, Adkins-Nappi model, global existence, scattering theory.}

\maketitle
\section{Introduction}
The classical nonlinear $\sigma$ model of nuclear physics (\cite{GL}, \cite{Gr1}, \cite{Gr2}), which describes interactions between nucleons and $\pi$ mesons, is characterized by the action
\EQ{\label{wm}S\,=\, \frac 12 \int g^{\mu\nu}S_{\mu\nu}\,dg\,=\,\frac 12 \int  g^{\mu\nu} \,\partial_\mu\phi^i\, \partial_\nu\phi^j\,
h_{ij}(\phi)\,dg,}
where $\phi :(\mathbb{R}^{1+3},g)\to(\mathbb{S}^3,h)$ is a map from the Minkowski spacetime, with $g= \text{diag}(-1,1,1,1)$, into the unit sphere of $\mathbb{R}^{4}$, endowed with the induced Riemannian metric, and $S_{\mu\nu}$ is the pullback metric corresponding to $\phi$. A priori, solutions for the associated Euler-Lagrange system (named also \emph{wave maps}) have the following energy norm conserved:  
\begin{equation}
E(\phi)\,=\,\frac 12 \,\int_{\mathbb{R}^n}\,|\phi_t|^2_h\,+\,|\nabla_x \phi|^2_h\ dx.
\label{en}
\end{equation}

One of the most interesting questions with respect to the wave maps system is whether smooth initial data of finite energy evolve into global regular solutions. Shatah \cite{S} gave a negative answer to this question by constructing a blow-up solution (later found by Turok and Spergel \cite{TS} in closed-form formula) in the category of degree-1 equivariant maps, i.e.,
\[
\phi(t,r,\psi, \theta)= (u(t,r),\psi, \theta), \qquad u(t,0)=0,  \qquad u(t,\infty)= \pi,
\]
which is given by taking
\EQ
{\label{ts}u(t,r)=2 \arctan\frac{r}{t}.}
This certifies the physical intuition of the nonlinear $\sigma$ model which says that, due to the attractive nature of the forces between $\pi$ mesons, degree-1 configurations would shrink to a point,  a degree-0 configuration, leading to a singularity formation. 

The models we are investigating, proposed by Skyrme \cite{S1}-\cite{S3}, respectively Adkins and Nappi \cite{AN}, are generalizations of the wave maps theory, which try to prevent the possible breakdown of the system in finite time. This is achieved, for example, in the case of the Adkins-Nappi model, by introducing a short range repulsion among the $\pi$ mesons, which is, in turn, accomplished by making the $\pi$ mesons interact with an $\omega$ vector meson. 

Therefore, from a mathematical point of view, one can predict for these problems that finite energy regular configurations are global in time (for a general discussion of this aspect, see \cite{MR2376667}). In proving such a claim, the standard approach has two major independent steps: the energy doesn't concentrate and small energy implies global regularity. Motivated also by the profile of blow-up solutions for wave maps, we focus our attention to degree-1 equivariant maps corresponding to the two models. In this case, non-concentration of energy has been established for the $3+1$ dimensional Adkins-Nappi model \cite{GR1, GR2}\footnote{Using similar methods, energy non-concentration has also been established for a $2+1$ dimensional Skyrme model \cite{GS}, while current work in progress \cite{GR3} investigates it for the problem in $3+1$ dimensions.}. In this article, we address the global issue.
 
We proceed next to introduce:
\subsection{The Skyrme model}
This theory is described by the Lagrangian
\EQ{ L\,=\,-\frac 12\, g^{\mu\nu}S_{\mu\nu} \,+\, \frac{\alpha^2}{4} \left(S^{\mu\nu} S_{\mu\nu}\,-\,(g^{\mu\nu}S_{\mu\nu})^2\right),}
where $S_{\mu\nu}$ is defined as in \eqref{wm} and $\alpha$ is a constant having the dimension of length. A degree-1 equivariant ansatz leads to the following nonlinear wave equation satisfied by the angular variable $u$:
\EQ{\label{sk} \left(1+\frac{2\al^2\sin^2u}{r^2}\right)(u_{tt}-u_{rr})-\frac{2}{r}u_r+\frac{\sin2u}{r^2}\left[1+\al^2\left(u_t^2-u_r^2+\frac{\sin^2u}{r^2}\right)\right]=0.}
The energy associated with \eqref{sk} is given by
\EQ{\mathcal{E}[u](t)=\int_0^\infty \left[\left(1 + \frac{2 \alpha^2\sin ^2 u}{r^2}\right)\frac{u_t^2+u_r^2}{2}+
\frac{\sin ^2 u}{r^2} +\frac{\alpha^2\sin^4 u}{2r^4}
\right]\,r^2 dr.\label{es}}

In recent work addressing the stability of the static solution, Bizon et al. \cite{BCR}  provided numerical support  for a finite energy global regularity claim. Also, at the time of our submission, we have learned of independent work of Li \cite{L}, which establishes global well-posedness for large $H^4$ data.
   
\subsection{The Adkins-Nappi model}
As mentioned before, this theory describes the interaction of $\pi$ mesons, represented by the classical nonlinear $\sigma$ model, with an $\omega$ meson,  which comes in the form of a gauge field $A=A_\mu dx^\mu$. The action of this model is 
\EQ{S\,=\, \int \left(\frac{1}{2}g^{\mu\nu}S_{\mu\nu}\,+\,\frac{1}{4} F^{\mu\nu} F_{\mu\nu} \right)\,dg\,-\, \int A_\mu j^\mu \, dt\,dx,} 
where $F_{\mu\nu}\,=\,\partial_\mu A_\nu - \partial_\nu A_\mu$ is the associated electromagnetic field, and $j$ is the baryonic current
\EQ{j^\mu\,=\,c\, \epsilon^{\mu\nu\rho\sigma}\,
\partial_\nu\phi^i \,\partial_\rho\phi^j \,\partial_\sigma\phi^k
\,\epsilon_{ijk},}
with $\epsilon$ being the Levi-Civita symbol and $c$ a normalizing constant. The equivariance assumption for $\phi$ and $A$ (see \cite{GR1}) yields, after scaling out some constant,
\EQ{ \label{an}
 u_{tt}-u_{rr}-\frac{2}{r}u_r+\frac{\sin 2u}{r^2}+\frac{(u-\sin u\cos u)(1-\cos 2u)}{r^4} =0,}
for which the conserved energy is 
\EQ{\mathcal{E}[u](t)=\int_0^\infty \left[\frac{u_t^2+u_r^2}{2}+
\frac{\sin ^2 u}{r^2} +\frac{(u-\sin u \cos u)^2}{2r^4}
\right]\,r^2 dr. \label{ean}}

We make now the following important
\begin{rem} (\cite{GR1}, \cite{GR3}) Finite energy smooth solutions of \eqref{sk} and \eqref{an} are uniformly bounded with
\EQ{\label{li} \|u\|_{L^\infty_{t,x}} \leq C(\mathcal{E}[u](0)),}
where $C(s)\to 0$ as $s\to 0$.
\label{rli}
\end{rem}

\section{Preliminaries}
We start this section by discussing what are the natural candidates for spaces where we should study the global well-posedness of the two problems. Usually, for equations that have a scaling associated to them, this is predicted by norms which leave the size of the initial data invariant. 

Both of our equations are not  scale-invariant. However, if we take advantage of Remark \ref{rli} and impose size restrictions on $u$, we can write formal asymptotics for certain expressions and obtain the following scale-invariant approximations for \eqref{sk}, respectively \eqref{an}:
\EQ{\label{ssk} \left(1+\frac{2\al^2u^2}{r^2}\right)(u_{tt}-u_{rr})-\frac{2}{r}u_r+\frac{2u}{r^2}\left[1+\al^2\left(u_t^2-u_r^2+\frac{u^2}{r^2}\right)\right]=0,}
\EQ{\label{san} 
u_{tt}-u_{rr}-\frac{2}{r}u_r+\frac{2u}{r^2}+\frac{u^5}{r^4} =0.}

A simple scale analysis given by 
\EQ{u_{\la}(t,r) = \la\, u(\frac t\la, \frac r\la), \qquad \|u_\la(0)\|_{\dot H^{5/2}(\R^3)} = \|u(0)\|_{\dot H^{5/2}(\R^3)},}
for \eqref{ssk} and
\EQ{u_{\la}(t,r) = \la^{1/2}\, u(\frac t\la, \frac r\la), \qquad \|u_\la(0)\|_{\dot H^2(\R^3)} = \|u(0)\|_{\dot H^2(\R^3)},}
for \eqref{san}, suggests a small data global well-posedness result in $\dot H^{5/2}(\R^3)$ for the Skyrme model and $\dot H^2(\R^3)$ for the Adkins-Nappi theory, where the homogeneous Sobolev norms are defined using the Fourier transform:
\[
\|u\|_{\dot H^s(\R^n)}\,=\,\||\xi|^s\,\hat{u}(\xi)\|_{L^2_\xi(\R^n)}.\]

We make the remark here that the energy in \eqref{es} and \eqref{ean} are bounded in $\dot H^{7/4} \cap \dot H^1(\R^3)$, respectively in $\dot H^{5/3} \cap \dot H^1(\R^3)$, hence the above argument suggests that both equations are supercritical with respect to the energy.

Next, we use the classical substitution $u=rv$ in order to transform \eqref{sk} and \eqref{an} into semilinear wave equations for $v$ on $\R^{5+1}$, which will be our main object of study from this point on.  Thus, \eqref{sk} becomes  
\EQ{ \label{skv}
 \pt v_{tt}-v_{rr}-\frac{4}{r}v_r + h_1(r,u)\,v^3 + h_2(r,u)\,v^5 + h_3(r,u)\,v^3v_r + h_4(r,u)\,v(v_t^2-v_r^2)=0,}
where 
\EQ{
  \pt h_1(r,u)=\frac{\sin2u - 2u}{(1+\frac{2\al^2\sin^2u}{r^2})\,u^3},
 \qquad h_2(r,u)=\frac{\al^2\sin2u\, (\sin^2u-u^2)}{(1+\frac{2\al^2\sin^2u}{r^2})\,u^5},
 \pr h_3(r,u)=\frac{4\al^2\sin u\,(\sin u - u\cos u)}{(1+\frac{2\al^2\sin^2u}{r^2})\,u^3}, 
 \pq h_4(r,u)=\frac{\al^2 \sin 2u}{(1+\frac{2\al^2\sin^2u}{r^2})\,u},}
while \eqref{an} changes to
\EQ{ \label{anv}
 v_{tt}-v_{rr}-\frac{4}{r}v_r+h_5(u)v^3+h_6(u)v^5=0,}
with 
\EQ{
 \pt h_5(u)=\frac{\sin 2u - 2u}{u^3},
 \qquad h_6(u)=\frac{(u-\sin u\cos u)(1-\cos 2u)}{u^5}.}

We can now state our main results that address the small data global well-posedness and scattering for both \eqref{skv} and \eqref{anv}.

\begin{thm}
There exists $\de>0$ such that for any radial intial data $(v(0,r),\p_t v(0,r))$ decaying as $r\to\I$ and satisfying 
\EQ{
 \|\p v(0,\cdot)\|_{\dot B^{3/2}_{2,1}\cap L^2(\R^5)}\le\de,}
the equation \eqref{skv} admits a unique global solution $v$ satisfying 
\EQ{
 \p v \in C(\R;\dot B^{3/2}_{2,1}\cap L^2(\R^5))\cap L^2(\R;\dot B^{3/4}_{4,1}\cap\dot B^{- 3/4}_{4,2}(\R^5))}
and for some $v_\pm$ solving the free wave equation, 
\EQ{
 \|\p(v-v_\pm)(t)\|_{\dot B^{3/2}_{2,1}\cap L^2(\R^5)} \to 0 \pq \text{as}\pq t\to\pm\I\, .}
\label{tsk}
\end{thm}

\begin{thm}
There exists $\de>0$ such that for any radial intial data $(v(0,r),\p_t v(0,r))$ decaying as $r\to\I$ and satisfying 
\EQ{
 \|\p v(0,\cdot)\|_{\dot H^1\cap L^2(\R^5)}\le\de,}
the equation \eqref{anv} admits a unique global solution $v$ satisfying 
\EQ{
 \p v \in C(\R;\dot H^1\cap L^2(\R^5))\cap L^2(\R;\dot B^{1/4}_{4,2}\cap\dot B^{-3/4}_{4,2}(\R^5))}
and for some $v_\pm$ solving the free wave equation, 
\EQ{
 \|\p(v-v_\pm)(t)\|_{\dot H^1\cap L^2(\R^5)} \to 0 \pq \text{as}\pq t\to\pm\I\, .}
 \label{tan}
\end{thm}

In both results, $\dot B^{s}_{p,q}(\R^n)$, for $s\in \R$, $1\leq p,q\leq \infty$, denotes the homogeneous Besov space whose norm is defined using a dyadic decomposition, i.e.,
\EQ{\label{B}
u\,=\,\sum_{\lambda\in 2^{\mathbb{Z}}} S_\lambda(\na) u, \qquad \|u\|_{\dot B^{s}_{p,q}(\R^n)}\,=\,\left(\sum_{\lambda\in 2^{\mathbb{Z}}} \left(\lambda^{s}\,\|S_\lambda(\na) u\|_{L^p(\R^n)}\right)^q\right)^{1/q},}
where $S_\la(\na):=\F^{-1}\chi(\la^{-1}\x)\F$ is the Fourier multiplier on $\R^n$, with a fixed radial $\chi\in C_0^\I(\R^n)$ satisfying 
\[
\supp\chi\subset\{1/2<|\x|<2\}, \qquad \sum_{\la\in 2^\Z}\chi(\la^{-1}\x)=1, \ (\forall)\x\not=0.\]

\begin{rem}
We note that the Sobolev norms in the above theorems simply translate to the ones for $u=rv$ on $\R^3$, while the Besov norms should be adjusted. Precisely: 
\EQ{
 \|v\|_{\dot H^s(\R^5)} \sim \|u\|_{\dot H^s(\R^3)}, 
 \pq \|v\|_{\dot B^s_{4,r}(\R^5)} \sim \|r^{-1/2}u\|_{\dot B^s_{4,r}(\R^3)}.}
\end{rem}

\begin{rem}
The global well-posedness in both theorems matches the one predicted by the formal asymptotics. The presence of two regularities in formulation is motivated by the fact that our equations are not scale-invariant and we want to work with homogeneous spaces.
\end{rem}

\begin{defn} In what concerns the notation for norms, from this point on, we usually work first on a fixed time estimate and only in the final stages we account also for the time variable. This will be clear from the context too.
\end{defn}

\section{Main argument}
We start by recording uniform bounds for the coefficients of the nonlinearities appearing in \eqref{skv} and \eqref{anv}. Using the notation 
\EQ{
 \ti h_i(u)\,=\,\CAS{\left(1+\frac{2\al^2\sin^2u}{r^2}\right)\,h_i(r,u), &1\leq i\leq 4, \\ h_i(u) &5\le i\le 6,}}
straightforward computations lead to:
\begin{lem}
$\ti h_i \,(1\leq i\leq 6)$ are all analytic functions of $u\in\R$; also, all are even with the exception of $\ti h_3$, which is odd. Moreover, $\ti h_1=\ti h_5\leq 0$, $\ti h_6\ge0$, and
\EQ{\label{123}|\ti h_1(u)|+ |\p_u^{j}\ti h_3(u)|\lec\LR{u}^{-2},
\pq |\p_u^{1+j}\ti h_1(u)|+ |\p_u^j\ti h_2(u)| \lec\LR{u}^{-3},}
\EQ{\label{46}|\p_u^j\ti h_4(u)|\lec\LR{u}^{-1}, \pq |\p_u^j\ti h_6(u)|\lec\LR{u}^{-4},}
for all $j\ge 0$, with $\LR{u}=(1+u^2)^{1/2}$.\label{hu}
\end{lem}

We prove first Theorem \ref{tan}, as the structure of the nonlinearities for \eqref{anv} is considerably simpler than the one for \eqref{skv}, needing only a standard Strichartz analysis. 

\subsection{Proof of Theorem \ref{tan}}
As we are using Strichartz estimates for the free wave equation on $\R^{5+1}$, it is sufficient to bound the nonlinearities in $L^1_t(\dot H^1\cap L^2)_x$. We begin by estimating the $L^1L^2$ norm. Relying on the Sobolev embeddings 
\EQ{\dot B^{1/4}_{4,2} \subset L^5,\pq \dot H^2 \subset \dot B^{3/4}_{4,2} \subset L^{10},}
we obtain
\EQ{\pt \|v^3\|_{L^1L^2} \lec \|v\|_{L^\infty L^{10}}\, \|v\|^2_{L^2L^5}\lec \|v\|_{L^\I \dot H^2} \,\|v\|_{L^2 \dot B^{1/4}_{4,2}}^2,
\pr \|v^5\|_{L^1L^2} \lec \|v\|^3_{L^\infty L^{10}}\, \|v\|^2_{L^2L^{10}}\lec \|v\|^3_{L^\I \dot H^2} \,\|v\|_{L^2 \dot B^{3/4}_{4,2}}^2,}
which, combined with the uniform bounds of Lemma \ref{hu} for $\ti h_5$ and $\ti h_6$, prove that $h_5(u)\,v^3$ and $h_6(u)\,v^5$ are bounded in $L^1L^2$.

Next, we treat the $L^1\dot H^1$ norms. For the main term, $h_6(u)v^5$, we have
\EQ{
 (h_6(u)v^5)_r=v^4v_r\,[5\,h_6(u)+u\,h_6'(u)]+v^4\frac{v}{r}\,[u\,h_6'(u)],}
which, based on \eqref{46} and Hardy's inequality, leads to
\EQ{
 \|h_6(rv)v^5\|_{\dot H^1_x} 
 \pt\lec \|v^4\|_{L^{5}}(\|v_r\|_{L^{10/3}}+\|v/r\|_{L^{10/3}}) 
 \pn\lec \|v\|_{L^{20}}^4 \|v_r\|_{L^{10/3}}.}
Using the Sobolev embeddings $\dot H^2 \subset \dot H^1_{10/3}\subset \dot B^{3/4}_{4,\I}$ and $\dot B^1_{4,1}\subset L^{20}$, together with the real interpolation $(\dot B^{3/4}_{4,\I},\dot B^{5/4}_{4,\I})_{1/2,1}=\dot B^1_{4,1}$, we deduce first
\EQ{
 \|v\|_{L^{20}} \lec \|v_r\|_{L^{10/3}}^{1/2}\|v\|_{\dot B^{5/4}_{4,\I}}^{1/2},}
which further implies
\EQ{
 \|h_6(rv)v^5\|_{L^1 \dot H^1} \pt\lec \|v_r\|_{L^\I L^{10/3}}^{3}\|v\|_{L^2\dot B^{5/4}_{4,\I}}^2 
 \pn\lec \|v\|_{L^\I \dot H^2}^3 \|v\|_{L^2 \dot B^{5/4}_{4,2}}^2.}

In what concerns the subcritical term $h_5(u)v^3$,  we proceed as above to derive the fixed time estimate
\EQ{
 \|h_5(u)v^3\|_{\dot H^1} \lec \|v^2\|_{L^{4}}\|v_r\|_{L^4} \lec \|v\|_{L^{10}}\|v\|_{L^{20/3}}\|v_r\|_{L^4}.}
Sobolev embeddings (e.g., $ \dot H^{7/4} \subset L^{20/3}$) allow us then to conclude that
\EQ{
 \|h_5(u)v^3\|_{L^1 \dot H^1} \lec \|v\|_{L^2 \dot B^{3/4}_{4,2}}\|v\|_{L^{\infty} \dot H^{7/4}}\|v\|_{L^2\dot B^1_{4,2}}.}

The rest of the proof is nothing but a  standard fixed point argument, in which one applies Strichartz estimates to the Duhamel formula. 

\subsection{Proof of Theorem \ref{tsk}}
In this case, the nonlinear terms are estimated in $L^1_t(\dot B^{3/2}_{2,1}\cap L^2)_x$ and we start by investigating the $L^1L^2$ norm. The cubic and quintic terms can be dispensed with immediately by the previous proof, as $|h_i(r,u)|\leq |h_i(u)|$. Also, using $|h_3|\lec 1/|u|$ (due to \eqref{123})  
and Hardy's inequality, we deduce
\EQ{ \|h_3(r,u)v^3v_r\|_{L^2} \lec \|v\|_{L^\I}\,\|v_r\|^2_{L^4},}
which, based on the Sobolev embedding $\dot B^{5/2}_{2,1}\subset L^\I$, gives
\EQ{\|h_3(r,u)v^3v_r\|_{L^1L^2} \lec \|v\|_{L^\I \dot B^{5/2}_{2,1}}\,\|\p v\|^2_{L^2 \dot B^0_{4,2}}.}
The last term is treated identically.

The most intricate part of this article is the analysis of  the $L^1\dot B^{3/2}_{2,1}$ norms. This is mainly due to the nonlinearities involving derivatives of $v$, which require a finer argument using spaces that can handle low regularity for high dimensional wave maps.

We notice first that the cubic and quintic terms can be treated simultaneously as we control the $\|v\|_{L^{\infty}_x}$ norm through $\dot B^{5/2}_{2,1}\subset L^\I$. Next, taking advantage of the $L^1\dot H^1$ analysis done in Theorem \ref{tan} and the real interpolation $(\dot H^1,\dot H^2)_{1/2,1}=\dot B^{3/2}_{2,1}$, we reduce this analysis to the study of 
\EQ{\left\| \p_r \left( \left(1+\frac{2\al^2\sin^2u}{r^2}\right)^{-1} \right) h_1(u)\, v^3 \right\|_{L^1L^2}
\pq\text{and}\pq \|h_1(r,u) v^3\|_{L^1\dot H^2}.} 
For the $L^1L^2$ norm, we use the elementary inequality
\EQ{\label{sin}
 \left|\frac{\sin u}{r}\right|^j \lec 1+\frac{2\al^2\sin^2u}{r^2}, \pq (\forall)\,0\leq j \leq 2,}
together with \eqref{123}, to derive
\EQ{\left|\p_r \left( \left(1+\frac{2\al^2\sin^2u}{r^2}\right)^{-1}\right) h_1(u) v^3\right| \lec \frac{v^2}{r}|v_r|  + (1+ |v|) \frac{v^2}{r^2},}
which implies
\EQ{\left\| \p_r \left( \left(1+\frac{2\al^2\sin^2u}{r^2}\right)^{-1} \right) h_1(u)\, v^3 \right\|_{L^1L^2}  \lec \left(1+\|v\|_{L^\I \dot B^{5/2}_{2,1}}\right)\,\|\p v\|^2_{L^2 \dot B^0_{4,2}}.} 
For the $L^1\dot H^2$ norm, a similar argument yields 
\EQ{|\p^2_{r}(h_1(r,u) v^3)| \lec (1+v^2)\left(\frac{v^2}{r^2}+\frac{v^2}{r}|v_r|+|v|v_r^2\right) + (1+|v|)v^2|v_{rr}|.}
Only the last term is not covered by previous estimates, being estimated as
\EQ{\|(1+|v|)v^2 v_{rr}\|_{L^1L^2} \pt \lec (1+\|v\|_{L^\I \dot B^{5/2}_{2,1}})\|v\|^2_{L^2L^{20}}\|v_{rr}\|_{L^\I L^{5/2}}
\pr \lec (1+\|v\|_{L^\I\dot B^{5/2}_{2,1}})\|v\|^2_{L^2\dot B^1_{4,2}}\|\p v\|_{L^\I \dot B^{3/2}_{2,1}},} 
which concludes the discussion of the cubic and quintic nonlinearities.

In what concerns the last two terms, we can no longer use the same approach because, in estimating their $L^1\dot H^2$ norm, we would have to deal with $v_{rrr}$, for which we do not have good bounds. For $h_3(r,u)v^3v_r$, we estimate its $L^1\dot B^{3/2}_{2,1}$ norm directly. We derive first  
\EQ{
 \|v^3v_r\|_{\dot B^{3/2}_{2,1}} \lec \|v\|_{L^\I}^3 \|v_r\|_{\dot B^{3/2}_{2,1}} + \|v^3\|_{\dot B^{3/2}_{10/3,1}}\|v_r\|_{L^5},}
which, based on the Sobolev embeddings $\dot B^{3/2}_{2,1} \subset L^5$ and $\dot B^{5/3}_{3,1}\subset \dot B^{3/2}_{10/3,1} \subset L^\I$,
leads to
\EQ{
 \|v^3v_r\|_{\dot B^{3/2}_{2,1}} \lec \|v\|^3_{\dot B^{5/3}_{3,1}}\, \|\p v\|_{\dot B^{3/2}_{2,1}}.}
Finally, we employ interpolation to deduce 
\EQ{ \|v\|_{\dot B^{5/3}_{3,1}} \lec \|v\|_{\dot B^{2}_{3,1}} + \|v\|_{\dot B^{1/2}_{3,2}}\lec \|v\|_{\dot B^{5/2}_{2,1}}^{1/3}\,\|v\|_{\dot B^{7/4}_{4,1}}^{2/3}+ \|v\|_{\dot H^{1}}^{1/3}\,\|v\|_{\dot B^{1/4}_{4,2}}^{2/3},}
allowing us to conclude
\EQ{
 \|v^3v_r\|_{L^1 \dot B^{3/2}_{2,1}}
 \lec  \|v\|_{L^\I_t(\dot B^{5/2}_{2,1}\cap\dot H^{1})_x}^2\, \|v\|_{L^2_t(\dot B^{7/4}_{4,1}\cap\dot B^{1/4}_{4,2})_x}^2.}
We claim that the coefficient $h_3(r,u)$ doesn't complicate things inside the norm and leave the details for the interested reader. 

All which is left to analyze is the critical term 
\EQ{
 N_4(r,v) := h_4(r,rv)\,v\,(v_t^2-v_r^2),} 
which distinguishes itself from the others by the presence of the null form,
\EQ{\label{q}
 Q(v,v)\,=\,v_t^2-v_r^2=-\square(v^2/2)+v\square v,}
where $\square=-\p^2_{t}+\p^2_{r}+\frac{4}{r}\p_r$ is the radial wave operator. We are exactly at the critical level of regularity for wave maps in $5+1$ dimensions (i.e., $\dot B^{5/2}_{2,1}$), and so we need to work with spaces which take into account \eqref{q}. 

Following Tataru \cite{T}, we introduce the space $F$ for $v$ as a function of $(t,x)\in\R^{n+1}$. Let $\chi\in C^\I(\R)$ be a smooth cutoff satisfying
\EQ{
 \supp\chi\subset(1/2,2) \pq \text{and} \pq \sum_{\la\in 2^\Z}\chi(\la^{-1}s)=1, \ (\forall) s\neq 0.}
For each $\la\in 2^\Z$, define $A_\la(D)$ and $B_\la(D)$ to be the Fourier multipliers in spacetime given by 
\EQ{
 A_\la(D)=\F^{-1}\,\chi(\la^{-1}|(\t,\x)|)\,\F, 
 \pq B_\la(D)=\F^{-1}\,\chi(\la^{-1}|\t^2-|\x|^2|/|(\t,\x)|)\,\F,}
where $\F$ denotes the Fourier transform in $(t,x)\in\R^{n+1}$. $A_\la$ and $B_\la$ are smooth projections to the spacetime frequencies situated at distance $\la$ from $(0,0)$, respectively the light cone $|\t|=|\x|$. Consider also 
\EQ{
 \ti B_\la(D)=\sum_{j\ge -4} B_{2^{-j}\la}(D).}
The function space $F$ on $\R^{n+1}$ is defined by the norm 
\EQ{
 \pt \|u\|_F = \sum_{\la\in 2^\Z} \la^{n/2}\|A_\la(D)u\|_{F_\la}, 
 \pq F_\la=X^{1/2}+Y_\la,
 \pr \|u\|_{X^b}=\sum_{\mu\in 2^\Z} \mu^{b}\|B_\mu(D)u\|_{L^2_{t,x}},
 \pq \|u\|_{Y_\la}=\|u\|_{L^\I L^2} + \la^{-1}\|\square u\|_{L^1L^2},}
with $\square=-\p_t^2+\De_x$. In \cite{T}, for $n\ge 4$, the following estimates were proved: 
\EQ{\label{F}
 \pt \|\p v\|_{L^\I \dot H^{n/2-1}} \lec \|\p v(0)\|_{\dot H^{n/2-1}}+\|\square v\|_{\square F}, 
 \pr \|v w\|_{F} \lec \|v\|_F \|w\|_F,
 \pq \|v w\|_{\square F} \lec \|v\|_F \|w\|_{\square F}, 
 \pr \|v_t^2-|\na v|^2\|_{\square F} \lec \|v\|_F^2,
}
where $\square F \subset L^1\dot B^{n/2-1}_{2,1}$ is defined by 
\EQ{
 \|u\|_{\square F} = \sum_{\la\in 2^\Z} \la^{n/2}\|A_\la(D)u\|_{\square F_\la},
 \pq \square F_\la=\la(X^{-1/2}+L^1L^2).}

In order to obtain $v\in L^\I_t(\dot H^{n/2}\cap \dot H^{n/2-1})_x$, we perform the iteration argument in the function space 
\EQ{ \label{def Z}
  Z:=\{v \in \S'(\R^{n+1})\mid v=v(t,x)=v(t,r),\ \|v\|_{F\cap|\na|F}<\I\},}
where $|\na|F$ is defined by 
\EQ{
 \|u\|_{|\na|F} = \sum_{\la\in 2^\Z} \la^{n/2-1}\|A_\la(D)u\|_{F_\la}.}
Since three of the nonlinear terms have already been estimated in $L^1_t(\dot H^{n/2-1}\cap\dot H^{n/2-2})_x$, only a discussion of  $N_4$ remains. More precisely, we need 
\EQ{\label{Z}
 \|N_4\|_{\square F \cap |\na|\square F} \lec \|v\|_Z^2,}
for small $v\in Z$. To take advantage of \eqref{F}, we can rewrite $N_4$ as:
\EQ{ \label{exp N_4}
 N_4\pt=(v_t^2-v_r^2)\,\frac{\al^2\sin 2u}{r}\,\frac{1}{1+\frac{2\al^2\sin^2u}{r^2}}
 \pr=(v_t^2-v_r^2)\,\frac{\al^2}{r}\,\sum_{j=0}^\I\frac{(-1)^j(2u)^{2j+1}}{(2j+1)!}\,\sum_{k=0}^\I (-2)^k\left(\frac{\al \sin u}{r}\right)^{2k},}
for small $\|v\|_{L^\I}$. Also, we will need the following

\begin{lem} \label{alg for u}
 If $n\ge 5$, then for any radial functions $v=v(t,r)$ and $w=w(t,r)$, 
\EQ{ \label{alg rZ}
 \|r v w\|_Z \lec \|v\|_Z \|w\|_Z.}  
 
 For $n\ge 4$ and general functions, $v=v(t,x)$ and $w=w(t,x)$, 
\EQ{ \label{D prod}
 \|vw\|_{|\na|F} \lec \|v\|_{|\na|F}\|w\|_F,
 \qquad \|vw\|_{|\na|\square F} \lec \|v\|_{|\na|F}\|w\|_{\square F}.}
 \end{lem}

\eqref{D prod} together with \eqref{F}  implies that $F\cap|\na|F$ is a Banach algebra, while the same is true for $rZ$, due to \eqref{alg rZ}. Therefore, for small $\|v\|_ Z$, we have that both $\sin u$ and the series in $j$ from  \eqref{exp N_4} are bounded in $rZ$, while the series in $k$ is bounded in $F\cap|\na|F$. It follows that $N_4$ is bounded in $\square F\cap|\na|\square F$, which completes the contraction estimate in $Z$ (i.e., \eqref{Z}) and so the proof of Theorem \ref{tsk}. Hence, we are left with
 
\begin{proof}[Proof of Lemma \ref{alg for u}]
In order to gain the $r^{-1}$ decay in \eqref{alg rZ}, we employ the radial Sobolev inequality in the sharp dyadic form. Using stationary phase estimates, we obtain 
\EQ{
 \|r^{(n-1)/2}S_\la(\na)\fy\|_{L^\I} \lec \la^{1/2}\|\fy\|_{L^2},}
for any radial $\fy\in L^2_x$, where $S_\la(\na)$ is the dyadic decomposition in $x$ used in \eqref{B}. This implies, after relying on interpolation and Sobolev embeddings, that 
\EQ{\label{pq}
 \|r^{\al(1/p-1/q)}S_\la(\na)\fy\|_{L^q} \lec \la^{(n-\al)(1/p-1/q)} \|\fy\|_{L^p}}
holds for all radial functions $\fy\in L^p_x$, $0\le\al\le n-1$, and $2\le p\le q$. 
Next, if one combines the Strichartz estimates in $F_\la$ (see Theorem 4 in \cite{T}) 
\EQ{
 \|v\|_{L^\I L^2}+\la^{-\frac{n+1}{2(n-1)}}\|v\|_{L^2 L^{\frac{2(n-1)}{n-3}}} \lec \|v\|_{F_\la}}
with \eqref{pq}, it follows that
\EQ{\label{rv}
 \la\|rv\|_{L^\I_{t,x}} + \la^{3/2}\|rv\|_{L^2 L^\I} \lec \la^{n/2}\|v\|_{F_\la}.}
We note that, in bounding the $L^2 L^\I$ norm, we need $n\ge 5$, as one applies \eqref{pq} for $\al= \frac{2(n-1)}{n-3} \le n-1$.

The rest of the proof follows in the same spirit with the one in \cite{T} for \eqref{F}. Since all spaces use $\ell^1$ summability over the dyadic decomposition, it suffices to prove our estimates for single dyadic pieces, which are denoted $v_\mu=A_\mu(D)v$ and $w_\la=A_\la(D)w$,  for $\mu,\la\in 2^\Z$.

In proving \eqref{alg rZ}, we may assume that $\mu\le\la$ by symmetry. We use the decomposition 
\EQ{w_\la=w_\la^{<\mu}+w_\la^{>\mu}, \qquad w_\la^{<\mu} = \ti B_\mu(D)w_\la.}
The support of $\F(v_\mu w_\la^{<\mu})$ is contained in the region
\EQ{
 |\t|+|\x| \lec \la, \qquad ||\t|^2-|\x|^2| \lec \mu\la,}
which is not changed by multiplication with $x$. If $v$ is radial, then \eqref{rv} implies
\EQ{
 \|xv_\mu w_\la^{<\mu}\|_{L^2_{t,x}}
 \lec \|rv_\mu\|_{L^2 L^\I} \|w_\la^{<\mu}\|_{L^\I L^2} 
 \lec \mu^{(n-3)/2}\|v_\mu\|_{F_\mu} \|w_\la\|_{F_\la},}
where we have relied also on the fact that $\ti B_\mu(D)A_\la(D)$ is bounded on $L^\I L^2$ (see the proof of Theorem 3 in \cite{T}). Using the above information on the Fourier support, we deduce 
\EQ{\label{x1}
 \|xv_\mu w_\la^{<\mu}\|_{X^{1/2}} \lec \mu^{n/2-1}\|v_\mu\|_{F_\mu} \|w_\la\|_{F_\la}.}

The other component, $w_\la^{>\mu}$, is nonzero only if $\mu\ll\la$. Therefore, if we multiply it by $v_\mu$, this does not essentially change the Fourier distance from $(0,0)$ and from the light cone of $w_\la^{>\mu}$. We may further decompose 
\EQ{
 w_\la^{>\mu} = \sum_{2^{4}\mu\,<\,\nu\,\le\, \la} w_\la^\nu + w_\la^0,}
such that $w_\la^\nu$ is supported in $||\ta|^2-|\x|^2|\sim \nu\la$ and 
\EQ{
 \|w_\la^{>\mu}\|_{F_\la} \sim \sum_\nu \|w_\la^\nu\|_{X^{1/2}} + \|w_\la^0\|_{Y_\la}.}
Again, by Theorem 3 in \cite{T}, we have 
\EQ{\label{w0}
 \|w_\la^0\|_{L^1L^2} \lec \mu^{-1}\|w_\la^0\|_{Y_\la}.}
The $X$ component of $xv_\mu w_\la^{>\mu}$ is estimated using \eqref{rv} by 
\EQ{\label{x2}
 \|xv_\mu \sum_\nu w_\la^\nu\|_{X^{1/2}} 
 \pt\lec \sum_\nu \nu^{1/2}\|rv_\mu w_\la^\nu\|_{L^2_{t,x}} 
 \pn\lec \|rv_\mu\|_{L^\I_{t,x}} \sum_\nu \nu^{1/2}\|w_\la^\nu\|_{L^2_{t,x}} 
 \pr\lec \mu^{n/2-1}\|v_\mu\|_{F_\mu}\|w_\la\|_{F_\la}.}
The $Y$ component is bounded by 
\EQ{\label{x3}
\|xv_\mu w_\la^0\|_{Y_\la}
 \pt\lec \|rv_\mu\|_{L^\I_{t,x}}\|w_\la^0\|_{L^\I L^2}
 + \la^{-1}\|\square(xv_\mu w_\la^0)\|_{L^1L^2}
 \pr\lec \mu^{n/2-1}\|v_\mu\|_{F_\mu}\|w_\la^0\|_{Y_\la} + \la^{-1}\|\square(xv_\mu w_\la^0)-xv_\mu\square w_\la^0\|_{L^1L^2},}
where the last term is controlled, due to \eqref{rv} and \eqref{w0}, by 
\EQ{
 \la^{-1} \mu\la \|rv_\mu\|_{L^\I_{t,x}} \|w_\la^0\|_{L^1L^2}
 \lec \mu^{n/2-1}\|v_\mu\|_{F_\mu}\|w_\la^0\|_{Y_\la}.}

Putting together \eqref{x1}, \eqref{x2}, and \eqref{x3}, we obtain 
\EQ{
 \|xv_\mu w_\la\|_{F\cap|\na|F} \lec \mu^{n/2-1}(\la^{n/2}+\la^{n/2-1})\|v_\mu\|_{F_\mu} \|w_\la\|_{F_\la},}
whose summation over $\mu,\la$ finishes the proof of \eqref{alg rZ}. 

For \eqref{D prod}, it is enough to check the high-high interaction $\mu\sim\la$, since otherwise the proof would follow from the corresponding versions without $|\na|$, already proved in \cite{T}. 

If $\mu\sim\la$, the support of $\F(v_\mu w_\la)$ is in the region $|\t|+|\x|\lec \mu$, and, as before, we deduce
\EQ{
 \|v_\mu w_\la\|_{L^2_{t,x}} \lec \|v_\mu\|_{L^2L^\I}\|w_\la\|_{L^\I L^2} \lec \mu^{-n/2+1/2}\cdot\mu^{n/2-1}\|v_\mu\|_{F_\mu} \la^{n/2}\|w_\la\|_{F_\la},}
which is bounded in $\ell^1_{\nu\lec\mu}(\nu^{-n/2+1}X^{1/2}_\nu)$. Similarly, using the  $\square F_\la\subset\la^{3/2}L^2_{t,x}$ embedding, we obtain 
\EQ{
 \|v_\mu w_\la\|_{L^1L^2} \lec \|v_\mu\|_{L^2L^\I} \|w_\la\|_{L^2_{t,x}}
 \pt\lec \mu^{n/2-1/2}\la^{3/2}\|v_\mu\|_{F_\mu}\|w_\la\|_{\square F_\la}
 \pr\lec \mu^{-n/2+2}\cdot \mu^{n/2-1}\|v_\mu\|_{F_\mu} \la^{n/2}\|w_\la\|_{\square F_\la},}
which is bounded in $\ell^1_{\nu\lec\mu}(\nu^{-n/2+2}(L^1L^2)_\nu)$. Thus, we conclude the proof of \eqref{D prod}. 
\end{proof}

\section*{Acknowledgements}
The first two authors would like to thank the Mathematisches Forschungsinstitut in Oberwolfach for the hospitality in fall 2010, where part of this project was completed. The first author was supported in part by the National Science Foundation Career grant DMS-0747656. The third author was supported in part by the Department of Energy contract DE-FG02-91ER40685.

\end{document}